\documentclass[a4paper,12pt]{amsart}
\usepackage[T1]{fontenc}    
\usepackage[utf8]{inputenc}
\usepackage{enumerate}
\usepackage{amssymb}
\usepackage{fullpage}
\usepackage{color}

\def\RR{\mathbb R}
\def\NN{\mathbb N}
\def\uu{\mathcal U}
\def\vv{\mathcal V}
\def\uuu{\overline{\mathcal U}}
\def\vvv{\overline{\mathcal V}}
\def\eps{\varepsilon}

\def\aa{\mathcal A}
\newcommand{\set}[1]{\left\lbrace #1\right\rbrace}
\providecommand{\abs}[1]{\left\lvert#1\right\rvert}

\newcommand{\remove}[1]{ }
\newcommand{\qtq}[1]{\quad\text{#1}\quad}

\newtheorem{theorem}{Theorem}
\newtheorem{lemma}[theorem]{Lemma}
\newtheorem{corollary}[theorem]{Corollary}

\theoremstyle{definition}

\theoremstyle{remark}
\newtheorem{remark}[theorem]{Remark}
\remove{}

\begin{document}
\title[]{Hausdorff distance of univoque sets}
\author{Yi Cai}
\address{School of Mathematical Sciences, East China Normal University,
Shanghai 200241, People’s Republic of China}
\email{52170601013@stu.ecnu.edu.cn}

\author{Vilmos Komornik}
\address{College of Mathematics and Computational Science,
Shenzhen University,
Shenzhen 518060,
People’s Republic of China,
and
Département de mathématique,
Université de Strasbourg,
7 rue René Descartes,
67084 Strasbourg Cedex, France}
\email{komornik@math.unistra.fr}
\subjclass[2010]{Primary 11A63; Secondary 11K55, 54E40}


\keywords{Non-integer base expansions;  Hausdorff metric; unique $\beta$-expansion.}
\date{Version 2019-06-23}
\thanks{This work has been done during the author's stay in the
Department of Mathematics of the University of Strasbourg.
He thanks the members of the department for their hospitality.}

\begin{abstract}
Expansions in non-integer bases have been investigated abundantly since their introduction by Rényi.
It was discovered by Erdős et al. that the sets of numbers with a unique expansion have a much more complex structure than in the integer base case.
The present paper is devoted to the continuity properties of these maps with respect to the Hausdorff metric.
\end{abstract}
\maketitle

Expansions in non-integer bases have been investigated abundantly since the pioneering works of Rényi \cite{R1957} and Parry \cite{P1960}.
It was discovered by Erdős et al. \cite{EHJ1991, EJK1990} that the sets of numbers with a unique expansion have a much more complex structure than in the integer base case.
The size, the topology and the Hausdorff dimension of these sets, and their dependence on the base has been clarified by the work of many authors \cite{DK1995,
KLP2003,
KL1998,
KL2007,
GS2001,
DK2009,
DK2011,
KL2015,
KLLD2016,
KKL2017,
AK2018},
and leading to a number of  unexpected results.

The present paper is devoted to the continuity properties of these maps with respect to the Hausdorff metric.
The topological results that we will use are contained in the papers \cite{DK2009,
DKL2016} that can also serve as an introduction to many aspects of this theory.

Fix a positive integer $M$.
By a \emph{sequence} we mean an element $c=(c_i)_{i=1}^{\infty}$ of the set $\set{0,1,\ldots,M}^{\NN}$.
By defining the distance $\rho((c_i),(d_i))$ of two different sequences as $2^{-n}$ if $c_n\ne d_n$, and $c_i=d_i$ for all $i<n$, $\set{0,1,\ldots,M}^{\infty}$ becomes a metric space.

Given a real number $q>1$, a sequence is called an \emph{expansion} of a real number $x$ in \emph{base} $q$ over the \emph{alphabet} $\set{0,1,\ldots,M}$ if
\begin{equation*}
(c_i)_q:=\sum_{i=1}^{\infty}\frac{c_i}{q^i}=x.
\end{equation*}
We denote by $\uu_q$ the set of real numbers having a unique expansion, and by $\uu_q'$ the set of corresponding expansions (sequences).
We recall the elementary relation
\begin{equation*}
p<q\Longrightarrow\uu_p'\subset\uu_q'.
\end{equation*}
We are interested in the continuity properties of the maps  $q\mapsto\uu_q'$ and $q\mapsto\uu_q$ where $q$ runs over the interval $(1,\infty)$ and the distances of the  image sets  are taken in Hausdorff's sense, i.e., $d_H(A,B)$ is the the infimum of the numbers $r>0$ such that
\begin{equation*}
A\subset\cup_{x\in B}B_r(x)\qtq{and}B\subset\cup_{x\in A}B_r(x);
\end{equation*}
here $B_r(x)$ denotes the open ball of radius $r$, centered at $x$.
Let us observe that
\begin{equation}\label{1}
d_H(\overline{A},\overline{B})=d_H(A,B)
\end{equation}
for any sets $A,B\subset\RR$ and their closures $\overline{A},\overline{B}$; this property remains valid in every metric space.

In order to state our theorem we need the sets $\uu$ and $\vv$ introduced in \cite{EJK1990} and \cite{KL2007}, respectively.
We denote by $\uu$ the set of bases in which $x=1$ has a unique expansion, and by $\vv$ the set of bases in which $x=1$ has a unique \emph{doubly infinite} expansion.
Here a sequence $(c_i)$ is called \emph{infinite} if it has infinitely many nonzero digit, and \emph{doubly infinite} if it is infinite, and  its \emph{reflection} $\overline{(c_i)}:=(M-c_i)$ is also infinite.
We recall that
\begin{equation*}
\uu\subsetneqq\uuu\subsetneqq\vv=\vvv,
\end{equation*}
and that both sets $\vv\setminus\uuu$ and $\uuu\setminus\uu$ are countably infinite.
If $M=1$, then the smallest element of $\vv\setminus\uuu$ is the Golden Ratio ($\approx 1.61803$) \cite{KL2007} and the smallest element of $\uu$ is the Komornik--Loreti constant ($\approx 1.78723$) \cite{KL1998}.
An example of an element of $\uuu\setminus\uu$ is the \emph{Tribonacci number} ($\approx 1.83929$), i.e., the positive root of the equation $q^3=q^2+q+1$.
The purpose of this paper is to prove the following theorem:

\begin{theorem}\label{t1}\mbox{}
\begin{enumerate}[\upshape (i)]
\item The functions $q\mapsto \uu_q'$ and $q\mapsto \uu_q$ are left continuous.
\item The functions $q\mapsto \vv_q'$ and $q\mapsto \vv_q$ are right continuous.
\item Each of the four functions is continuous at $q\in(1,\infty)$ if and only if $q\notin\vv\setminus\uuu$.
\end{enumerate}
\end{theorem}

We start by studying the maps $q\mapsto \uu_q'$ and $q\mapsto \vv_q'$.

\begin{lemma}\label{l2}
The function $q\mapsto\uu_q'$ is left continuous.
\end{lemma}

\begin{proof}
For any set $\aa$ of sequences and any  integer $n\ge 1$ we denote by $B_n(\aa)$ the set of initial words of length $n$ of the sequences in $\aa$, i.e.,
\begin{equation*}
B_n(\aa):=
\set{c_1\cdots c_n\ :\ (c_i)\in\aa}.
\end{equation*}
For the left continuity of the map at some given $q\in(1,\infty)$, it suffices to establish the following property:
\emph{For each integer $n\ge 1$ there exists a $p_n<q$ such that $B_n(\uu_{p_n}')=B_n(\uu_q')$}.
If there exists a $p<q$ such that $\uu_p'=\uu_q'$, then we may obviously choose $p_n:=p$ for every $n$.

We distinguish three cases.
If  $q\in (1,\infty)\setminus\vv$, then, since $\vv$ is closed, $q$ has a neighborhood $(a,b)$ such that $\uu_p'=\uu_q'$ for all $p\in (a,b)$.
Indeed, we may choose the connected component of $(1,\infty)\setminus\vv$ that contains $q$.

If $q\in\vv\setminus\uuu$, then $\uu_p'=\uu_q'$ for all $p\in (a,q]$ with $a:=\max\set{r\in\vv\ :\ r<q}$ by \cite[Theorem 1.7]{DK2009}.

If $q\in\uuu$, then by the proof of \cite[Lemma 3.12]{DKL2016} and by Remark 3.13 following that proof there exists for each $n\ge 1$ a $t_n\in\vv\setminus\uuu$ satisfying
\begin{equation}\label{2}
t_n<q\qtq{and}\alpha_1(t_n)\cdots\alpha_n(t_n)=\alpha_1(q)\cdots\alpha_n(q).
\end{equation}
Choose $p_n\in(t_n,q)$ arbitrarily.
If $(c_i)\in\uu_q'\setminus\uu_{p_n}'$, then there exists a smallest integer $j\ge 0$ such that
\begin{align*}
&c_j<M\qtq{and}c_{j+1}\cdots c_{j+n}\ge\alpha_1(q)\cdots\alpha_n(q)
\intertext{or}
&c_j>0\qtq{and}c_{j+1}\cdots c_{j+n}\ge\overline{\alpha_1(q)\cdots\alpha_n(q)}.
\end{align*}
In fact,  we have equality here because $(c_i)\in\uu_q'$.
Assume by symmetry that $c_{j+1}\cdots c_{j+n}=\alpha_1(q)\cdots\alpha_n(q)$, and define $(d_i):=c_1\cdots c_j\alpha(t_n)$, then $d_1\cdots d_{j+n}=c_1\cdots c_{j+n}$ and hence $d_1\cdots d_n=c_1\cdots c_n$.
We complete the proof by showing that $(d_i)\in\uu_{p_n}'$.

We have to show that
\begin{align*}
&(d_{k+i})<\alpha(p_n)\qtq{whenever}d_k<M
\intertext{and}
&(d_{k+i})>\overline{\alpha(p_n)}\qtq{whenever}d_k>0.
\end{align*}
For $k<j$ this follows from the minimality of $j$ by \eqref{2}.

For $k\ge j$ this follows from the relation $p_n\in\vv$:
\begin{align*}
&(d_{k+i})=\alpha_{k-j+i}(t_n)\alpha_{k-j+i+1}(t_n)\cdots
\le \alpha(t_n)<\alpha(p_n)
\intertext{and}
&(d_{k+i})=\alpha_{k-j+i}(t_n)\alpha_{k-j+i+1}(t_n)\cdots
\ge \overline{\alpha(t_n)}>\overline{\alpha(p_n)}.\qedhere
\end{align*}
\end{proof}

\begin{lemma}\label{l3}
The function $q\mapsto\vv_q'$ is right continuous.
\end{lemma}

\begin{proof}
Using the notations of the preceding proof, for the right continuity of the map at some given $q\in(1,\infty)$ it suffices to establish the following property:
\emph{For each integer $n\ge 1$ there exists a $t_n>q$ such that $B_n(\vv_{t_n}')=B_n(\vv_q')$}.
If there exists a $t>q$ such that $\vv_t'=\vv_q'$, then we may obviously choose $t_n:=t$ for every $n$.

We distinguish four cases.
The case $q=M+1$ is obvious because if $t\ge M+1$, then $\vv_t'$ is the set of all infinite sequences and therefore $B_n(\vv_t')$ contains all words of length $n$.

If  $q\in (1,\infty)\setminus\vv$, then, as in the proof of Lemma \ref{l2}, $q$ has a neighborhood $(a,b)$ such that $(a,b)\cap\vv=\varnothing$, and therefore $\vv_t'=\uu_t'=\uu_q'=\vv_q'$ for all $t\in (a,b)$.

If $q\in\vv\setminus\uu$, then $\vv_t'=\vv_q'$ for all $t\in [q,b)$ with $b:=\min\set{p\in \vv\ :\ p>q}$.

Finally, let $q\in\uu$ and $q<M+1$.
Assume for the moment that there exists an $r_n\in \uuu\setminus\uu$ such that\begin{equation}\label{3}
r_n>q
\qtq{and}
\alpha_1(r_n)\cdots\alpha_n(r_n)=\alpha_1(q)\cdots\alpha_n(q).
\end{equation}
Given $(c_i)\in\uu_{r_n}'\setminus\uu_q'$ arbitrarily, we consider the smallest integer $j\ge 0$ such that
\begin{align*}
&c_j<M\qtq{and}c_{j+1}\cdots c_{j+n}\ge\alpha_1(r_n)\cdots\alpha_n(r_n)
\intertext{or}
&c_j>0\qtq{and}c_{j+1}\cdots c_{j+n}\ge\overline{\alpha_1(r_n)\cdots\alpha_n(r_n)}.
\end{align*}
In fact, then we have equality because $(c_i)\in\uu_{r_n}'$.

Assume by symmetry that
\begin{equation*}
c_{j+1}\cdots c_{j+n}=\alpha_1(r_n)\cdots\alpha_n(r_n),
\end{equation*}
and define $(d_i):=c_1\cdots c_j\alpha(q)$.
Then $d_1\cdots d_{j+n}=c_1\cdots c_{j+n}$ and hence $d_1\cdots d_n=c_1\cdots c_n$.
We complete the proof by showing that $(d_i)\in\uu_{r_n}'$.

We have to show that
\begin{align*}
&(d_{k+i})<\alpha(r_n)\qtq{whenever}d_k<M
\intertext{and}
&(d_{k+i})>\overline{\alpha(r_n)}\qtq{whenever}d_k>0.
\end{align*}
For $k<j$ this follows from the minimality of $j$ by \eqref{3}.

For $k\ge j$ this follows from the relation $q\in\vv$:
\begin{align*}
&(d_{k+i})=\alpha_{k-j+1}(q)\alpha_{k-j+2}(q)\cdots
\le \alpha(q)<\alpha(r_n)
\intertext{and}
&(d_{k+i})=\alpha_{k-j+1}(q)\alpha_{k-j+2}(q)\cdots
\ge \overline{\alpha(q)}>\overline{\alpha(r_n)}.
\end{align*}
It remains to prove the existence of $r_n$.
First we show that $q$ is a right accumulation point of $\uuu$.
For otherwise there exists a $\delta>0$ such that $(q,q+\delta)\cap\uuu=\varnothing$.
This means that $(q,q+\delta)\subset(q_0,q_0^*)$ for some connected component $(q_0,q_0^*)$ of $(1,\infty)\setminus\uuu$, and hence $q=q_0=1$ or $q=q_0\in\uuu\setminus\uu$, contradicting our assumption that $q\in\uu$.
Since $\uuu\setminus\uu$ is dense in $\uuu$, it follows that $q$ is also a right accumulation point of $\uuu\setminus\uu$.
There exists therefore a sequence $(s_k)\subset\uuu\setminus\uu$ satisfying $s_k\searrow q$.
Then $\alpha(s_k)$ converges pointwise to $\beta(q)$ (see \cite{DK2011}), and $\beta(q)=\alpha(q)$ because $q\in\uu$.
Therefore, for every fixed $n\ge 1$ there exists a sufficiently large $k$ such that $\alpha(s_k)$ and $\alpha(q)$ have the same initial word of length $n$.
\end{proof}

\begin{remark}\label{r4}
It would have been tempting to construct a suitable $r_n$ satisfying \eqref{3} as in the proof of \cite[Lemma 3.10]{DKL2016}, by writing
$\alpha(q)=(\alpha_i)$ and taking
\begin{equation*}
\alpha(r_n):=\left( \alpha_1\cdots\alpha_m\ \overline{\alpha_1\cdots\alpha_{m-1}(\alpha_m-1)\alpha_1\cdots\alpha_m}\right) ^{\infty}
\end{equation*}
with a sufficiently large $m$ satisfying $\alpha_m>0$.
However, we do not get a point satisfying
$r_n>q$ in general.
For example, if $M=1$ and $(\alpha_i)=1(10)^{\infty}$, then $\alpha(r_n)<(\alpha_i)$ for every $m\ge 2$ satisfying $\alpha_m=1$ by a simple computation because they start with the same word of length $m$, but
\begin{equation*}
\alpha_{m+1}(r_n)\alpha_{m+2}(r_n)
=00<01=\alpha_{m+1}\alpha_{m+2}.
\end{equation*}
\end{remark}

\begin{corollary}\label{c5}\mbox{}
\begin{enumerate}[\upshape (i)]
\item If $q_n\nearrow q$, then $\uu_{q_n}'\to\uu_q'$ and $\vv_{q_n}'\to\uu_q'$.
\medskip
\item If $q_n\searrow q$, then $\uu_{q_n}'\to\vv_q'$ and $\vv_{q_n}'\to\vv_q'$.
\end{enumerate}
\end{corollary}

\begin{proof}
The  relations $\uu_{q_n}'\to\uu_q'$ and $\uu_{q_n}'\to\vv_q'$ are equivalent to Lemmas \ref{l2} and \ref{l3}.
The two other relations hence follow because
\begin{equation*}
p<r<s\Longrightarrow \uu_p'\subset\vv_r'\subset\uu_s'
\qtq{and}
\vv_p'\subset\uu_r'\subset\vv_s'
\end{equation*}
by the lexicographic characterizations of these sets \cite{DK2009}.
\end{proof}

Corollary \ref{c5} remains valid for $\uu_q$ and $\vv_q$ by the following result:

\begin{lemma}\label{l6}
For each fixed $q>1$, the function $(c_i)\mapsto (c_i)_q$ is (uniformly) continuous from $\set{0,1,\ldots,M}^{\NN}$ to $\RR$.
\end{lemma}

\begin{proof}
Consider two sequences $c=(c_i)$ and $d=(d_i)$.
If $\rho((c_i),(d_i))<2^{-m}$, then $c_1\cdots c_m=d_1\cdots d_m$, and the following estimate holds:
\begin{equation*}
\abs{(c_i)_q-(d_i)_q}
=\abs{\sum_{i=1}^{\infty}\frac{c_i}{q^i}-\sum_{i=1}^{\infty}\frac{d_i}{q^i}}
=\abs{\sum_{i=m+1}^{\infty}\frac{c_i-d_i}{q^i}}
\le \frac{M}{q^m(q-1)}.
\end{equation*}

Now for any given $\eps>0$, choose $m\ge 1$ such that $\frac{M}{q^m(q-1)}<\eps$, and then a $\delta>0$ such that $\delta<2^{-m}$.
Then
\begin{equation*}
\rho((c_i),(d_i))<\delta\Longrightarrow \abs{(c_i)_q-(d_i)_q}<\eps.\qedhere.
\end{equation*}
\end{proof}

Now we are ready to prove our theorem.

\begin{proof}[Proof of Theorem \ref{t1}]
In view of the general property \eqref{1} of the Hausdorff distance, Corollary \ref{c5}  implies the following:
\begin{itemize}
\item the map $q\mapsto \uu_q'$ is left continuous, and it is right continuous at $q$ if and only if $\overline{\uu_q'}=\overline{\vv_q'}$;
\item the map $q\mapsto \vv_q'$ is right continuous, and it is left continuous at $q$ if and only if $\overline{\uu_q'}=\overline{\vv_q'}$.
\end{itemize}
Combining Corollary \ref{c5} and Lemma \ref{l6} we obtain the analogous results for the other two maps:
\begin{itemize}
\item the map $q\mapsto \uu_q$ is left continuous, and it is right continuous at $q$ if and only if $\overline{\uu_q}=\overline{\vv_q}$;
\item the map $q\mapsto \vv_q$ is right continuous, and it is left continuous at $q$ if and only if $\overline{\uu_q}=\overline{\vv_q}$.
\end{itemize}
We complete the proof of the theorem by proving the equivalence of the following three properties:
\begin{equation*}
\overline{\uu_q'}=\overline{\vv_q'},
\quad \overline{\uu_q}=\overline{\vv_q}
\qtq{and}q\notin\vv\setminus\uuu.
\end{equation*}
We distinguish three cases.

If $q\notin\vv$, then $\uu_q'=\vv_q'$ by \cite[Theorem 1.5]{DK2009}, and hence also  $\uu_q=\vv_q$, so that $\overline{\uu_q'}=\overline{\vv_q'}$ and $\overline{\uu_q}=\overline{\vv_q}$.

If $q\in\uuu$, then $\overline{\uu_q}=\vv_q$ by \cite[Theorem 1.3 (i)]{DK2009} and therefore $\overline{\uu_q}=\overline{\vv_q}$.
Moreover, the \emph{proof} of \cite[Theorem 1.3 (i)]{DK2009} shows the stronger property $\overline{\uu_q'}=\vv_q'$, so that we have also $\overline{\uu_q'}=\overline{\vv_q'}$.

Finally, if $q\in\vv\setminus\uuu$, then $\uu_q$ and $\vv_q$ are different closed sets by \cite[Theorem 1.4 (i), (ii)]{DK2009}, and then $\uu_q'$ and $\vv_q'$ are also different closed sets by Lemma \ref{l6}, so that $\overline{\uu_q'}\ne\overline{\vv_q'}$ and $\overline{\uu_q}\ne\overline{\vv_q}$.
\end{proof}

\end{document}